\documentclass[11pt,a4paper]{amsart}
\usepackage[utf8]{inputenc}
\usepackage[english]{babel}
\usepackage{a4wide}

\usepackage{amsaddr}
\usepackage{amsmath, amssymb}
\usepackage{amsthm}
\usepackage{cleveref}
\usepackage{xcolor}

\usepackage{enumitem}
\setlist[enumerate]{label=(\roman*)}

\usepackage{todonotes,url}

\newtheorem{theorem}{Theorem}[section]
\newtheorem*{theorem*}{Theorem}

\newtheorem{proposition}[theorem]{Proposition}
\newtheorem{corollary}[theorem]{Corollary}

\newtheorem{lemma}[theorem]{Lemma}

\theoremstyle{definition}
\newtheorem{definition}[theorem]{Definition}

\newtheorem*{thm*}{Theorem}

\newtheorem*{claim*}{Claim}

\newtheorem{remark}[theorem]{Remark}
\newtheorem{example}[theorem]{Example}

\newtheorem{question}[theorem]{Question}
\newtheorem*{question*}{Question}

\newcommand{\R}{\mathbb{R}}

\newcommand{\Z}{\mathbb{Z}}
\newcommand{\N}{\mathbb{N}}

\newcommand{\F}{\mathcal{F}}

\newcommand{\fTilde}{\widetilde{f}}

\newcommand{\eps}{\varepsilon}

\newcommand{\fBar}{\overline{f}}

\newcommand{\FBar}{\overline{F}}
\newcommand{\lo}{\operatorname{span}}
\newcommand{\ord}{\operatorname{ord}}
\newcommand{\ncl}{\operatorname{ncl}}
\newcommand{\im}{\operatorname{Im}}

\newcommand{\abs}[1]{\left| #1 \right|}
\newcommand{\norm}[1]{\left\| #1 \right\|}
\newcommand{\supnorm}[1]{\norm{#1}_{\infty}}
\newcommand{\eqNote}[1]{\stackrel{\text{#1}}{=}}

\newcommand{\Lipz}{\operatorname{Lip}_0}
\newcommand{\absFunc}{\abs{\: \cdot \: }}
\newcommand{\inv}{\operatorname{Inv}}

\newcommand{\ignore}[1]{}

\newtheorem*{acknowledgement}{Acknowledgement}

\title{Almost-invariant elements of group actions on $\Lipz$ spaces}
\author{Tom\'a\v s Raunig}
\email{raunig@karlin.mff.cuni.cz}
\address{Faculty of Mathematics and Physics, Charles University, Sokolovská 83, Prague, 186 00, Czech Republic}
\thanks{The author was supported by the Research grant GAČR 23-04776S}
\date{}

\begin{document}
	
	\begin{abstract}
		This work is motivated by a~question published in E. Glasner's paper \emph{On a question of Kazhdan and Yom Din} regarding the possibility to approximate functionals on a~Banach space which are almost invariant with respect to an~action of a~discrete group by functionals that are invariant. We study the case when the Banach space is a~Lipschitz-free space equipped with an~action induced by an~action by isometries on the underlaying space. We find a~few different conditions sufficient for the answer to be positive; for example the case of free or finitely presented groups endowed with left-invariant metrics acting on themselves by translations. Relations to quasimorphisms are also briefly studied.
	\end{abstract}
	
	\maketitle
	
	\section{Introduction}
	
	The motivation for the present work comes from the paper~\cite{Glasner22}, where a~question posed by Kazhdan and Yom Din is presented and studied. Before stating the question, let us give a~definition fixing the notation used in the paper. Let $G$ be a~group, $X$ be a~Banach space and let $\alpha: G \times X \to X$ be an~action by linear isometries. We define the \emph{dual action} $\alpha^*: G \times X^* \to X^*$ by
	\begin{equation*}
		\alpha^*(g, x^*)(x) = x^*(\alpha(g^{-1}, x)), \quad x\in X, x^* \in X^*, g \in G.
	\end{equation*}
	Or, written using the shorthand notation, $gx^*(x) = x^*(g^{-1}x)$.
	We say that $x^* \in X^*$ is \emph{$G$-invariant} if $gx^* = x^*$ for all $g \in G$. If the group is clear from the context, we may merely say that $x^*$ is invariant.
	
	Now we can state the aforementioned question.
	\begin{question} \label{que:KYD}
		Let $\delta > 0$, $X$ be a~Banach space and let $G$ be a~discrete group acting on $X$ by linear isometries. Suppose that $x^* \in S_{X^*}$ satisfies that $\norm{gx^* - x^*} \leq \frac{\delta}{10}$ for all $g \in G$. Must there exist $G$-invariant $y^* \in X^*$ with $\norm{x^* - y^*} \leq \delta$?
	\end{question}
	
	Notice that the question bears similarities to the well known and widely studied Kazhdan's property (T) (for a~book concerning this property see, e.g., \cite{BHV08}). To simplify notation we introduce
	\begin{definition}
		Let $X$ be a~Banach space, $G$ a~group acting on $X$ by linear isometries and $\delta > 0$. We say that $x^* \in X^*$ is \emph{$\delta$-invariant} (for the action of $G$) if
		\begin{equation*}
			\forall g \in G \colon \norm{gx^* - x^*} \leq \delta.
		\end{equation*}
	\end{definition}
	
	As mentioned, the problem has already been studied in~\cite{Glasner22}. There, it has been shown that the answer in general is negative, but examples where the answer is positive were also provided. These positive cases are the space $X = \ell_1(G)$ with the action induced by translations and actions of amenable groups (and hence, in particular, abelian groups) on any Banach space.
	
	Seeing that the question is not black or white, it is reasonable to ask how will things turn out for specific pairs of groups and Banach spaces. In this paper, we will consider the case $X = \F(M)$, that is, our Banach space is the Lipschitz-free space over some metric space. Note that this choice is related to the aforementioned example of $\ell_1(G)$. For example, if $(M, d) = (\Z, \abs{\cdot})$, then $\F(M)$ is linearly isometric to $\ell_1(\Z)$. More details can be found in~\cite{WeaverLA}.
	
	Let us recall the characteristic properties of Lipschitz-free spaces. Given a~pointed metric space $(M, d, 0)$, there is a~Banach space $\F(M)$ (which is unique up to linear isometry) and an~isometric embedding $\delta: M \to \F(M)$ such that
	\begin{enumerate}
		\item $\delta(0) = 0$;
		\item $\delta(M \setminus \{0\})$ is linearly independent and $\F(M) = \overline{\lo} \ \delta(M)$;
		\item $\delta$ is an~isometry;
		\item if $Y$ is a~Banach space and~$f \in \Lipz(M, Y)$, then there exists a~linear map $T_f: \F(M) \to Y$ satisfying $f = T_f \circ \delta$ and $\norm{T_f} = L(f)$.
	\end{enumerate}
	
	Let us establish what our actions look like. We begin with some action by isometries of a~group $G$ on a~pointed metric space $(M, d, 0)$ (note that we do not require the action to fix the distinguished point $0$). Define $\alpha: G \times \F(M) \to \F(M)$ by setting $\alpha(g, \cdot)$ to be the linearization of the mapping $x \mapsto \delta(gx) - \delta(g0)$ (which is given by the condition (iv) above). Working through the definitions, one then easily checks that this is an~action by linear isometries on $\F(M)$.
	
	It is well-known that $\F(M)^* = \Lipz(M)$ with the norm on the space $\Lipz(M)$ given by the Lipschitz number, i.e. $\norm{f} = L(f), f \in \Lipz(M)$. For this and other results on Lipschitz-free spaces one may see e.g.~\cite{WeaverLA}. Note that by $\norm{\cdot}$ on $\Lipz(M)$ we will always mean this norm. Using this fact, we arrive to an~equivalent formulation of Question~\ref{que:KYD} in the context described above:
	\begin{question} \label{que:KYDLip}
		Let $\delta > 0$, $(M, d, 0)$ be a~pointed metric space and $G$ be a~discrete group acting on $M$ by isometries. Consider the action of $G$ by linear isometries $\alpha: G \times \Lipz(M) \to \Lipz(M)$ given by
		\begin{equation*}
			\alpha(g, f)(x) = f(g^{-1}x) - f(g^{-1}0), \quad g \in G, f \in \Lipz(M), x \in M.
		\end{equation*}
		Assume $f \in \Lipz(M)$ with $L(f) = 1$ is $\delta/10$-invariant with respect to the action~$\alpha$, that is $\norm{\alpha(g, f) - f} = \norm{x \mapsto f(g^{-1}x) - f(x)} \leq \delta/10$ for every $g \in G$.
		Must there exist invariant $\fBar \in \Lipz(M)$ with $\norm{f - \fBar} \leq \delta$?
	\end{question}
	Here it is appropriate to acknowledge that not every isometric action on a~Lipschitz-free space must arise from an~isometric action on the underlaying space as described above. However, for some metric spaces, it does turn out to be the case as was shown in~\cite{CDT24}.
	
	From now on, whenever we have an~action of a~group on a~pointed metric space $M$, we will, unless explicitly stated otherwise, consider $\Lipz(M)$ to be equipped with the action of the group given in the question above. Similarly, we will prefer the shortened notation $gf$ for the action of an~group element $g$ on a~function $f$.
	
	 To end this section, let us briefly go over the structure of this paper and mention some of the results. We note that the formulations below are simplified as not to be too technical. For their full versions, see the appropriate parts of the paper. In Section~\ref{section:Aux} we prepare some technical tooling for the following sections and, among others, we provide more tangible equivalent conditions for ($\delta$-)invariance of a~function. One such, which demonstrates the nature of the problem rather nicely, is
	\begin{theorem*} [\Cref{cor:LipInv}]
		Let $(M, d, 0)$ be a~pointed metric space, $G$ be a~group acting by isometries on $M$, and assume $f \in \Lipz(M)$. Then $f$ is invariant if and only if there is a~homomorphism $H: G \to \R$ such that for any $x \in M$ and $g \in G$ it holds $f(gx) = f(x) + H(g)$.
	\end{theorem*}
	
	In \Cref{sect:main} we provide a~positive answer in case of actions of free-groups on themselves equipped with word-length metrics:
	\begin{theorem*}[\Cref{thm:KYDfreeWord}]
		Let $\delta > 0$ and $F_S$ be a~free group with generating set $S$ equipped with the word metric and action by left translations. Let $f \in \Lipz(F_S)$ be $\delta$-invariant. Then there is an~invariant element $\fBar \in \Lipz(F_S)$ with $\norm{f - \fBar} \leq \delta/2$. Moreover, for any invariant $\fTilde \in \Lipz(F_S)$ we have $\norm{f-\fBar} \leq \norm{f-\fTilde}$.
	\end{theorem*}
	
	If we replace the free group with a~finitely presented group, we also get a~result in the positive direction, but with the caveat that a~constant depending on the group appears.
	\begin{theorem*} [\Cref{thm:KYDfinPres}]
		Let $G$ be a~finitely presented group equipped with the word metric and action by left translations. Then there exists a~constant $C > 0$ depending on $G$ such that for any $\delta > 0$ and $f \in \Lipz(G)$ $\delta$-invariant there is $\fBar \in \inv_G(G)$ with $\norm{f - \fBar} \leq C\delta$.
	\end{theorem*}
	
	On the other end of the spectrum, if the metric significantly shrinks the distances along the orbits of our group, we get the sought after approximation assuming the metric space has favourable structure.
	\begin{theorem*} [\Cref{thm:badGroup}]
		Let $G$ be a~group acting by isometries on a~pointed metric space $(M, d, 0)$. Assume the following conditions:
		\begin{enumerate}
			\item $\{g^n x : n \in \N\}$ is not bi-Lipschitz equivalent to $(\Z, \abs{\cdot})$ for every pair $g \in G, x \in X$;
			\item there is $\alpha \geq 1$ and a~fundamental domain $D$ such that $d(x, y) \leq \alpha d(G x, G y), x, y \in D$.
		\end{enumerate}
		Then for every $\delta > 0$ and $\delta$-invariant $f \in \Lipz(M)$ there is an~invariant $\fBar \in \Lipz(M)$ with $\norm{f - \fBar} \leq (2\alpha + 1)\delta$.
	\end{theorem*}
	For the definition of fundamental domain, see the discussion after~\Cref{thm:badGroup}.
	
	Finally, in \Cref{sect:quasimorphisms}, we see that the notion of almost-invariant function coincides with that of a~partial quasimorphism and use our results to recover part of~\cite[Theorem 1.1]{Kedra}.

	\begin{acknowledgement}
		I would like to thank Michal Doucha for his invaluable insights and comments regarding this paper.
	\end{acknowledgement}

	\section{Auxiliary results} \label{section:Aux}
	
	The aim of this section is to shed some light on (almost-)invariant functionals in the dual of Lipschitz-free spaces, that is, (almost-)invariant Lipschitz functions vanishing at the basepoint. We will provide a~few characterisations of such functions which will find ample use in later sections. Towards the end of the section we show there is a~reasonable way to talk about the ``growth of the function in the~direction of some element of the group''.
	
	Our first task is to find conditions equivalent to the definition of $\delta$-invariance more suitable in the context of Lipschitz functions.
	
	\begin{lemma} \label{lma:LipDeltaInv}
		Let $(M, d, 0)$ be a~pointed metric space and $G$ be a~group acting on $M$ by isometries. Let $\delta > 0$ and $f \in \Lipz(M)$. Then the following conditions are equivalent:
		\begin{enumerate}
			\item $f$ is $\delta$-invariant;
			\item $f - gf$ is $\delta$-Lipschitz for every $g \in G$;
			\item the mapping $x \mapsto f(gx) - f(x)$ is $\delta$-Lipschitz for every $g \in G$;
			\item for any $x, y \in M$ and $g_1, \dots g_n \in G$ holds
			\begin{align*}
				\abs{f(g_1 \cdots g_n x) - f(x) - \sum_{i=1}^{n} (f(g_iy) - f(y))}
				\leq \delta \left(nd(x,y) + \sum_{i=2}^{n} d(g_ix, x)\right).
			\end{align*}
		\end{enumerate}
	\end{lemma}
	\begin{proof}
		$(i) \iff (ii)$: Recall from~\Cref{que:KYDLip} that a~function $f \in \Lipz(M)$ is $\delta$-invariant if and only if $\norm{f - gf} \leq \delta$ for all $g \in G$. Since the norm on $\Lipz(M)$ we use is the Lipschitz number, the equivalence follows.
		
		$(ii) \iff (iii)$: This equivalence follows from the equality
		\begin{align*}
			&\abs{(g^{-1}f - f)(x) - (g^{-1}f - f)(y)}\\
			&\quad= \abs{f(gx) - f(g0) - f(x) - \big( f(gy) - f(g0) - f(y) \big)}\\
			&\quad= \abs{\big( f(gx) - f(x) \big) - \big( f(gy) - f(y) \big)}
		\end{align*}
		which is valid for any $x, y \in M$ and $g \in G$.
		
		$(iii) \iff (iv)$: Notice that for $n = 1$, $(iv)$ reduces to $(iii)$. So $(iv) \implies (iii)$ and $(iii)$ directly implies $(iv)$ for $n=1$. If $n > 1$, we expand $f(g_1 \cdots g_n x) - f(x)$ into the sum
		\begin{equation*}
			f(g_1 \cdots g_n x) - f(x)
			= \sum_{i=1}^n f(g_1 \cdots g_ix) - f(g_1 \cdots g_{i-1}x),
		\end{equation*}
		where by $f(g_1 \cdots g_{i-1}x)$ for $i = 1$ we understand as just $f(x)$. Now we may write
		\begin{align*}
			&\abs{f(g_1 \cdots g_n x) - f(x) - \sum_{i=1}^{n} (f(g_ix) - f(x))}\\
			&\quad \leq \sum_{i=1}^{n}\abs{\big(f(g_1 \cdots g_ix) - f(g_1 \cdots g_{i-1}x)\big) - \big(f(g_ix) - f(x)\big)}\\
			&\quad = \sum_{i=2}^{n}\abs{\big(f(g_1 \cdots g_{i-1} g_ix) - f(g_ix)\big) - \big(f(g_1 \cdots g_{i-1}x) - f(x)\big)}\\
			&\hspace{3em}+ \abs{f(g_1x) - f(g_1x) - (f(x) - f(x))}\\
			&\quad = \sum_{i=2}^{n}\abs{\big(f(g_1 \cdots g_{i-1} g_ix) - f(g_ix)\big) - \big(f(g_1 \cdots g_{i-1}x) - f(x)\big)}.
		\end{align*}
		For each summand, we use $(iii)$ for $g' = g_1 \cdots g_{i-1}$, $x' = g_i x$ and $y' = x$ to obtain
		\begin{align*}
			\abs{f(g_1 \cdots g_n x) - f(x) - \sum_{i=1}^{n} \big(f(g_ix) - f(x)\big)}
			\leq \sum_{i=2}^n \delta d(g_i x, x).
		\end{align*}
		Finally,
		\begin{align*}
			&\abs{f(g_1 \cdots g_n x) - f(x) - \sum_{i=1}^{n} \big(f(g_iy) - f(y)\big)} \\
			&\ = \abs{f(g_1 \cdots g_n x) - f(x) - \sum_{i=1}^{n} \big(f(g_ix) - f(x) - f(g_ix) + f(x) + f(g_iy) - f(y)\big) }\\
			&\ \leq \sum_{i=2}^n \delta d(g_i x, x) + \sum_{i=1}^n \abs{\big(f(g_ix) - f(x) \big) - \big(f(g_iy) - f(y) \big)}\\
			&\ \leq \delta \left(\sum_{i=2}^n d(g_i x, x) + nd(x,y)\right),
		\end{align*}
		where in the last inequality we have again used $(iii)$ for each summand.
	\end{proof}
	
	Using the previous lemma, we obtain analogous equivalent conditions for invariance.
	\begin{corollary} \label{cor:LipInv}
		Let $(M, d, 0)$ be a~pointed metric space, $G$ be a~group acting by isometries on $M$ and assume $f \in \Lipz(M)$. Then the following are equivalent:
		\begin{enumerate}
			\item $f$ is invariant under the action of $G$;
			\item $f-gf$ is constant for every $g \in G$;
			\item $x \mapsto f(gx) - f(x)$ is constant for every $g \in G$;
			\item there is a~homomorphism $H: G \to \R$ such that for any $x \in M$ and $g \in G$ holds $f(gx) = f(x) + H(g)$.
		\end{enumerate}
	\end{corollary}
	\begin{proof}
		First, note that from the definitions it immediately follows that $f$ is invariant if and only if $f$ is $\delta$-invariant for all $\delta > 0$.
		
		$(i) \implies (ii), (iii)$: If $f$ is invariant, then it is $\delta$-invariant for all $\delta > 0$ and the conditions $(ii)$ and $(iii)$ of~\Cref{lma:LipDeltaInv} are satisfied for all $\delta > 0$, implying $(ii)$ and $(iii)$.
		
		$(ii), (iii) \implies (i)$: If $(ii)$ or $(iii)$ holds, then the corresponding condition in~\Cref{lma:LipDeltaInv} holds for all $\delta > 0$. Hence $f$ is $\delta$-invariant for all $\delta > 0$ and thus also invariant.
		
		$(iii) \implies (iv)$: Define $H: G \to \R$ by $H(g) = f(g0), g \in G$. For any $x \in X$ and $g \in G$ we use $(iii)$ to deduce $f(gx) - f(x) = f(g0) - f(0) = H(g)$, which is the desired equality. We just need to show that $H$ is a~homomorphism. To that end, we again use $(iii)$: for any $g, h \in G$ we have
		\begin{equation*}
			H(gh) = f(gh0) = f(gh0) - f(h0) + f(h0) - f(0) = f(g0) - f(0) + f(h0) - f(0) = H(g) + H(h).
		\end{equation*}
		
		$(iv) \implies (iii)$: Fix $g \in G$. For any $x, y \in M$ we have
		\begin{equation*}
			\big(f(gx) - f(x)\big) - \big(f(gy) - f(y) \big) = \big(f(x) + H(g) - f(x)\big) - \big( f(y) - H(g) - f(y) \big) = 0.
		\end{equation*}
	\end{proof}
	
	Later, the case of a~group acting on itself equipped with a~left-invariant metric by translations will be of interest. For that, we have
	
	\begin{corollary} \label{cor:InvForLeftTranslation}
		Let $G$ be a~group endowed with a~left-invariant metric. Then $f \in \Lipz(G)$ is invariant with respect to left-translations if and only if $f: G \to \R$ is a~homomorphism.
	\end{corollary}
	\begin{proof}
		By~\Cref{cor:LipInv}, $f$ is invariant if and only if there is a~homomorphism $H: G \to \R$ such that $f(gx) = f(x) + H(g)$ for every $x,g \in G$.
		
		If $f$ is a~homomorphism, then for every $g,x \in G$ we have $f(gx) = f(x) + f(g)$, so we may take $H = f$.
		
		On the other hand, if there is a~homomorphism $H: G \to \R$ such that $f(gx) = f(x) + H(g)$ for every $x,g \in G$, then for any $g \in G$ it holds $f(g) = f(ge) = f(e) + H(g) = H(g)$, so $f = H$ and thus $f$ is a~homomorphism.
	\end{proof}
	
	Next, we define a~quantity representing the mean growth in a~``direction'' of an~element of the group.
	
	\begin{definition}
		Let $G$ be a~group acting on a~metric space $(M, d)$ by isometries and $f \in \Lipz(M)$. Define $c^*_f: G \times M \to \R$, $* \in \{+, -, \emptyset\}$, by
		\begin{align*}
			c^+_f(s, x) &= \sup\limits_{g \in G} f(gsx) - f(gx), \hspace{2cm}
			c^-_f(s, x) = \inf\limits_{g \in G} f(gsx) - f(gx), \\
			c_f(s, x) &= \frac{c^+_f(s,x) + c^-_f(s, x)}{2}.
		\end{align*}
	\end{definition}
	
	If the mapping $f$ is clear from the context, we omit the subscript.
	
	\begin{remark} \label{rem:everythingIsDeltaInv}
		In the definition above, we do not explicitly require $f$ to be $\delta$-invariant for some $\delta$. But since we do require $G$ to act by isometries, we automatically obtain that for every $g \in G$ it is the case that $\norm{f - gf} \leq \norm{f} + \norm{gf} = 2\norm{f}$, i.e. $f$ is $2\norm{f}$-invariant. Note that this is not a~special feature of $\Lipz$ spaces and holds true for any action by isometries on any Banach space.
	\end{remark}
	
	\begin{proposition}\label{prop:meanProps}
		The mappings $c^+_f, c^-_f, c_f$ are well-defined for every $f \in \Lipz(M)$ and if $f$~is moreover $\delta$-invariant for some $\delta > 0$, then for every $x \in M$ and $s \in G$ it holds
		\begin{equation*}
			c_f(s, x) - \frac{\delta}{2} d(sx, x) \leq c^-_f(s, x) \leq c_f(s, x) \leq c^+_f(s, x) \leq c_f(s, x) + \frac{\delta}{2} d(sx, x)
		\end{equation*}
		and $c_f(s, x) = -c_f(s^{-1}, x)$.
	\end{proposition}
	\begin{proof}
		We start by showing that the mappings are well-defined. By the remark above, we may assume that $f$ is $\delta$-invariant for some $\delta > 0$. Let $s \in G$ and $x \in M$. Then
		\begin{align*}
			c^+(s, x)
			&= \sup\limits_{g \in G} \big(f(gsx) - f(gx) - f(sx) + f(x) \big) + f(sx) - f(x)\\
			&\leq \sup\limits_{g \in G} \big( (f(gsx) - f(sx)) - (f(gx) - f(x)) \big) + \norm{f}d(sx,x)\\
			&\leq  (\delta + \norm{f})d(sx,x) < \infty.
		\end{align*}
		Similarly, one deduces $c^+(s, x) \geq -(\delta + \norm{f})d(sx,x) > -\infty$ and the same estimates for $c^-_f$.

		Let $\eps > 0$ and find $g, h \in G$ such that $c^+(s, x) \leq f(gsx) - f(gx) + \eps$ and $c^-(s, x) \geq f(hsx) - f(hx) - \eps$. We estimate
		\begin{align*}
			c^+(s,x) - c^-(s,x)
			&\leq f(gsx) - f(gx) - f(hsx) + f(hx) + 2\eps\\
			&= \big( f(gh^{-1} (hsx)) - f(hsx) \big) - \big( f(gh^{-1} (hx)) - f(hx) \big) + 2\eps \\
			&\leq \delta d(hsx, hx) + 2\eps \\
			&= \delta d(sx, x) + 2\eps.
		\end{align*}
		As $\eps > 0$ was arbitrary, we obtain $\abs{c^+(s,x) - c^-(s,x)} = c^+(s,x) - c^-(s,x) \leq \delta d(sx, x)$. Using the properties of the mean, we obtain the required inequalities.
		
		It remains to show that $c(s, x) = -c(s^{-1}, x)$. To that end, calculate
		\begin{align*}
			c^+(s, x)
			&= \sup_{g \in G} f(gsx) - f(gx)
			\eqNote{$g \rightsquigarrow gs^{-1}$} \sup_{g \in G} f(gs^{-1}sx) - f(gs^{-1}x)
			= \sup_{g \in G} -\big(f(gs^{-1}x) - f(gx)\big)\\
			&= -\inf_{g \in G} f(gs^{-1}x) - f(gx)
			= -c^-(s^{-1}, x).
		\end{align*}
		Analogously, one may obtain $c^-(s, x) = -c^+(s, x)$. Hence
		\begin{equation*}
			c(s, x) = \frac{c^+(s, x) + c^-(s, x)}{2} = \frac{-c^-(s^{-1}, x) - c^+(s^{-1}, x)}{2} = -c(s^{-1}, x).
		\end{equation*}
	\end{proof}
	
	Again, for the case of a~group acting on itself, we prepare
	
	\begin{proposition} \label{prop:deltaInvEquivToDeltaC}
		Let $G$ be a~group equipped with a~left-invariant metric and the action by left translations on itself and let $\delta > 0$. Then $f \in \Lipz(G)$ is $\delta$-invariant if and only if
		\begin{equation*}
			\forall s \in G: c_f^+(s, e) - c_f^-(s, e) \leq \delta d(s, e).
		\end{equation*}
	\end{proposition}
	\begin{proof}
		If $f$ is $\delta$-invariant, the claim follows from \Cref{prop:meanProps} by setting $x = e$.
		
		To prove the other implication, by \Cref{lma:LipDeltaInv}, it is enough to check that $x \mapsto f(g^{-1}x) - f(x)$ is $\delta$-Lipschitz. So, let $g,x,y \in G$. We have
		\begin{align*}
			&c_f^+(s, e) = \sup_{h\in G} f(hse) - f(he) = \sup_{h\in G} f(hs) - f(h)\\
			&c_f^-(s, e) = \inf_{h\in G} f(hse) - f(he) = \inf_{h\in G} f(hs) - f(h).
		\end{align*}
		We estimate
		\begin{align*}
			\big(f(gx) - f(x)\big) - \big(f(gy) - f(y)\big)
			&= \big(f(gx) - f(gy)\big) - \big(f(x) - f(y)\big)\\
			&=\big(f(gyy^{-1}x) - f(gy)\big) - \big(f(yy^{-1}x) - f(y)\big)\\
			& \leq c_f^+(y^{-1}x, e) - c_f^-(y^{-1}x, e) \leq \delta d(y^{-1}x, e) = \delta d(x, y).
		\end{align*}
		The estimate by $-\delta d(x, y)$ is obtained by switching the roles of $c_f^+(y^{-1}x, e)$ and $c_f^-(y^{-1}x, e)$.
	\end{proof}

	\section{Main results} \label{sect:main}
	
	\begin{lemma} \label{lma:LNforHomo}
		Let $G$ be a~group equipped with a~left-invariant metric $d$ and the action on itself by left-translations. If $h: G \to \R$ is a~homomorphism, then
		\begin{equation*}
			L(h) = \sup_{e \neq g \in G} \frac{\abs{h(g)}}{d(g, e)}.
		\end{equation*}
	\end{lemma}
	\begin{proof}
		By definition of the Lipschitz number, $\sup_{e \neq g \in G} \frac{\abs{h(g)}}{d(g, e)} \leq L(h)$. To show the converse inequality, let $g, g' \in G$, $g \neq g'$ and calculate
		\begin{equation*}
			\frac{\abs{h(g) - h(g')}}{d(g, g')}
			= \frac{\abs{h(g^{-1}g')}}{d(g, g')}
			= \frac{\abs{h(g^{-1}g')}}{d(g^{-1}g', e)}
			\leq \sup_{e \neq k \in G} \frac{\abs{h(k)}}{d(k, e)}.
		\end{equation*}
		Taking the supremum over all $g \neq g'$ yields the inequality.
	\end{proof}
	
	\begin{theorem} \label{thm:KYDfreeWord}
		Let $\delta > 0$ and $G = F_S$ be a~free group with generating set $S$ equipped with the word metric and action by left-translations. Let $f \in \Lipz(F_S)$ be $\delta$-invariant. Then, denoting $c(s) = c_f(s, e)$ for $s \in G$, the mapping $\fBar: F_S \to \R$ defined by
		\begin{equation*}
			\fBar(g) = \sum_{i=1}^n c(s_i),
		\end{equation*}
		where $n \in \N$ and $s \in (S^{\pm1})^n$ are such that $g = s_1 \cdots s_n$ is the reduced word representing $g$, is an~invariant element of $\Lipz(F_S)$ with $\norm{f - \fBar} \leq \delta/2$. Moreover, for any invariant $\fTilde \in \Lipz(F_S)$ we have $\norm{f-\fBar} \leq \norm{f-\fTilde}$.
	\end{theorem}
	\begin{proof}
		The mapping $\fBar$ is a~homomorphism and hence, by \Cref{lma:LNforHomo}, it is Lipschitz. By \Cref{cor:InvForLeftTranslation}, it is invariant. Let $g, h \in G$ and $g^{-1}h = s_1\cdots s_n$ be the reduced word representing $g^{-1}h$. Then $-\fBar(g) + \fBar(h) = \fBar(g^{-1}h) = \sum_{i=1}^n c(s_i)$. We have
		\begin{align*}
			\abs{(f-\fBar)(g) - (f-\fBar)(h)}
			&= \abs{f(g) - f(gg^{-1}h) - \sum_{i=1}^n c(s_i)}\\
			&= \abs{f(g) - f(gs_1\cdots s_n) - \sum_{i=1}^n c(s_i)}
		\end{align*}
		Defining $s_0 = e$ and using Proposition~\ref{prop:meanProps} for $x_i = g s_0 \cdots s_{i-1}$, we have
		\begin{align*}
			\abs{(f-\fBar)(g) - (f-\fBar)(h)}
			&\leq \sum_{i=1}^n \abs{f(g s_0 \cdots s_{i-1}) - f(g s_0 \cdots s_{i}) - c(s_i)}
			\leq \sum_{i=1}^n \frac{\delta}{2}
			= n\frac{\delta}{2}.
		\end{align*}
		
		To show optimality of $\fBar$, assume for contradiction that $\norm{\fBar - f} > \norm{\fTilde - f}$. There are $x, s \in G$ such that
		\begin{equation*}
			\frac{\abs{f(xs) - f(x) - c(s)}}{d(s,e)}
			= \frac{\abs{f(xs) - f(x) - \fBar(s)}}{d(s,e)}
			> \norm{\fTilde - f}
			\geq \sup_{g \in G} \frac{\abs{f(gs) - f(g) - \fTilde(s)}}{d(s,e)}.
		\end{equation*}
		Observe that by definition of $c(s)$, we have $\abs{f(xs) - f(x) - c(s)} \leq (c^+(s) - c^-(s))/2$.
		Now, we obtain an~estimate
		\begin{align*}
			\frac{c^+(s) - c^-(s)}{d(s,e)}
			&= \sup_{g \in G} \frac{f(gs) - f(g)}{d(s,e)} - \inf_{g \in G} \frac{f(gs) - f(g)}{d(s,e)}\\
			&= \sup_{g \in G} \frac{f(gs) - f(g) - \fTilde(s)}{d(s,e)} - \inf_{g \in G} \frac{f(gs) - f(g) - \fTilde(s)}{d(s,e)}\\
			&\leq 2\sup_{g \in G} \frac{\abs{f(gs) - f(g) - \fTilde(s)}}{d(s,e)}\\
			&< 2 \frac{\abs{f(xs) - f(x) - c(s)}}{d(s,e)}
			\leq \frac{c^+(s) - c^-(s)}{d(s,e)}.
		\end{align*}
		Because the inequality in the second to last step is strict, we have arrived to a~contradiction.
	\end{proof}
	
	The constant in \Cref{thm:KYDfreeWord} is optimal as demonstrated by the following example.
	\begin{example}
		Let $\delta > 0$ and consider $F_{\{1\}} = \Z$ equipped with the word length metric, i.e. $d(x, y) = \abs{x-y}, x,y\in\Z$. Define $f \in \Lipz(\Z)$ by $f(x) = 0, x \leq 0$ and $f(x) = \delta x, x > 0$. Then $f$ is $\delta$-invariant and $\fBar(x) = \frac{\delta}{2}x$ and $\norm{f-\fBar} = \delta/2$.
	\end{example}
	\begin{proof}
		Clearly, $f$ is $\delta$-Lipschitz. For $s \geq 0$ we have $c^+(s, 0) = \delta$ and $c^-(s, 0) = 0$. Conversely, for $s < 0$ we have $c^+(s, 0) = 0$ and $c^-(s, 0) = -\delta$. Either way, $c^+(s, 0) - c^-(s, 0) = \delta$ and so, by \Cref{prop:deltaInvEquivToDeltaC}, $f$ is $\delta$-invariant. Moreover, $c(1, 0) = \delta/2$, so $\fBar(x) = \frac{\delta}{2}x$ and we can calculate
		\begin{equation*}
			\norm{f - \fBar} = \norm{x \mapsto \frac{\delta}{2}x} = \frac{\delta}{2}.
		\end{equation*}
	\end{proof}
	
	Later on we will need finer control over the exact values of the invariant function by which we approximate. To enable this, we have
	
	\begin{corollary} \label{cor:KYD_FreeAdj}
		Let $\delta > 0$ and $F_S$ be a~free group with free generating set $S$. Equip $F_S$ with the word metric and the action on itself by left-translations. Let $f \in \Lipz(F_S)$ be $\delta$-invariant. Let $u \in \ell_\infty(S)$ and $\eta \geq 0$ satisfy $\abs{c_f(s, e) - u(s)} \leq \eta, s \in S$. Define $\fBar: F_S \to \R$ by
		\begin{equation*}
			\fBar(s_1^{a_1} \dots s_n^{a_n}) = \sum_{i=1}^{n} a_i u(s_i), \quad n \in \N, s \in S^n, a \in \{-1,1\}^n.
		\end{equation*}
		Then $\fBar \in \inv_{F_S}(F_S)$ and $\norm{f - \fBar} \leq \frac{\delta}{2} + \eta$.
	\end{corollary}
	\begin{proof}
		The mapping $\fBar$ is a~homomorphism, because $F_S$ is a~free group. Denote $\fTilde(g) = \sum_{i=1}^n a_i c_f(s_i, e)$ the function from \Cref{thm:KYDfreeWord} (here, we consider only $s_i \in S$, hence the need for the coefficients $a_i$; the functions are indeed the same by \Cref{prop:meanProps}).
		For $g \in G$ and its representation as a~reduced word $g = s_1^{a_1} \cdots s_n^{a_n}$ we calculate
		\begin{equation*}
			\frac{\abs{\fTilde(g) - \fBar(g)}}{d(g, e)}
			\leq \frac{1}{n}\sum_{i=1}^n \abs{a_i (u(s_i) - c_f(s_i, e))}
			\leq \frac{1}{n} \sum_{i=1}^n \eta
			= \eta.
		\end{equation*}
		From Lemma~\ref{lma:LNforHomo} we obtain that $\norm{\fTilde - \fBar} \leq \eta$. This means that $\fBar \in \Lipz(F_S)$ and, as it is a~homomorphism, by Corollary~\ref{cor:InvForLeftTranslation} we have $\fBar \in \inv_{F_S}(F_S)$. Moreover, 
		\begin{equation*}
			\norm{f - \fBar} \leq \norm{f - \fTilde} + \norm{\fTilde - \fBar} \leq \frac{\delta}{2} + \eta.
		\end{equation*}
	\end{proof}
	
	Every group can be written as a~quotient of a~free group. Since we already have a~result for free groups, we would like to transfer it to the quotients.
	
	\begin{proposition} \label{prop:KYD_onFactor}
		Let $S$ be any set, $F_S$ a~free group over $S$ equipped with the word metric, $N \trianglelefteq F_S$ and $\delta > 0$. Denote $G = F_S / N$, $q: F_S \to G$ the quotient map and equip $G$ with the word metric induced by the set $q(S) = \{q(s) \colon s \in S\}$. Equip both $F_S$ and $G$ with actions by left-translations. Let $f \in \Lipz(G)$ be $\delta$-invariant. Denote $F: F_S \to \R$, $F = f \circ q$. Then $F \in \Lipz(F_S)$, $F$~is $\delta$-invariant and for $\eta > 0$ the following conditions are equivalent
		\begin{enumerate}
			\item there exists $\fBar \in \inv_G(G)$ with $\norm{f - \fBar} < \eta$;
			\item there exists $\FBar \in \inv_{F_S}(F_S)$ with $\norm{F - \FBar} < \eta$ satisfying $N \subset \ker \FBar$.
		\end{enumerate}
		Moreover, for any $s \in F_S$ we have $c_f(q(s), e) = c_F(s, e)$.
	\end{proposition}
	\begin{proof}
		We will denote both metrics as $d$ as it will always be obvious which metric we refer to. First, note that $d(q(g), q(h)) \leq d(g, h)$ for any $g, h \in F_S$: if $d(g, h) = n$ and $g^{-1}h = s_1 \cdots s_n$ for some $s \in (S^{\pm1})^n$, then $q(g^{-1}h) = q(s_1)\cdots q(s_n)$ and hence $d(q(g), q(h)) \leq n$.
		Clearly $F(e) = f(q(e)) = 0$ and for any $g, h \in G$ holds
		\begin{equation*}
			\abs{F(g) - F(h)} = \abs{f(q(g)) - f(q(h))} \leq \norm{f} d(q(g), q(h)) \leq \norm{f} d(g, h).
		\end{equation*}
		So, $F \in \Lipz(F_S)$. Using \Cref{lma:LipDeltaInv}, we show that $F$ is $\delta$-invariant: for any $g, h, k \in G$ we have
		\begin{align*}
			&\abs{F(gh) - F(h) - \big(F(gk) - F(k)\big)}\\
			&\hspace{5em}= \abs{f(q(g)q(h)) - f(q(h)) - \big(f(q(g)q(k)) - f(q(k))\big)}\\
			&\hspace{5em}\leq \delta d(q(h), q(k))
			\leq \delta d(h, k).
		\end{align*}
		
		$(i) \implies (ii)$: Put $\FBar = \fBar \circ q$. Then for $g \in N$ holds $\FBar(g) = \fBar(q(g)) = \fBar(e_G) = 0$. This shows $N \subset \ker \FBar$ and in particular $\FBar(e) = 0$. As before, for any $g, h \in F_S$ we get
		\begin{align*}
			\abs{(F - \FBar)(g) - (F - \FBar)(h)}
			&= \abs{(f - \fBar)(q(g)) - (f - \fBar)(q(h))} \\
			&< \eta d(q(g), q(h))
			\leq \eta d(g,h),
		\end{align*}
		i.e. $\norm{F - \FBar} < \eta$ and hence also $\FBar \in \Lipz(F_S)$. For $g,h \in G$ holds $\FBar(gh) = \fBar(q(gh)) = \fBar(q(g)q(h))$. Since $\fBar$ is invariant, by Corollary~\ref{cor:InvForLeftTranslation} it is a~homomorphism and so $\FBar(gh) = \fBar(q(g)) + \fBar(q(h)) = \FBar(g) + \FBar(h)$. Corollary~\ref{cor:InvForLeftTranslation} now implies that $\FBar$ is invariant, because it is a~homomorphism.
		
		$(ii) \implies (i)$: $\FBar$ is invariant and hence a~homomorphism by Corollary~\ref{cor:InvForLeftTranslation}. Since $\ker q = N \subseteq \ker \FBar$, $\FBar$ factors through $q$ to a~homomorphism $\fBar: G \to \R$ with $\FBar = \fBar \circ q$. To show $\fBar \in \Lipz(G)$, choose $g, h \in F_S$ and $s_1, \dots, s_n \in S^{\pm 1}$ such that $q(g^{-1}h) = q(s_1) \cdots q(s_n)$ and $d(q(g), q(h)) = n$. Then
		\begin{align*}
			\abs{\fBar(q(g)) - \fBar(q(h))}
			&= \abs{\fBar(q(g^{-1}h))}
			= \abs{\sum_{i=1}^n \fBar(q(s_i))}
			= \abs{\sum_{i=1}^n \FBar(s_i)}\\
			&\leq n \norm{\FBar}
			= \norm{\FBar} d(q(g), q(h)).
		\end{align*}
		So we have $\fBar \in \Lipz(G)$ and since it is a~homomorphism, by Corollary~\ref{cor:InvForLeftTranslation} we also have $\fBar \in \inv_G(G)$. It remains to show that $\norm{f - \fBar} < \eta$. Choose $q(g), q(h) \in G$ and $s_1, \dots, s_n \in S^{\pm 1}$ such that $q(g^{-1}h) = q(s_1) \cdots q(s_n)$ and $d(q(g), q(h)) = n$. This implies that the word $s_1 \cdots s_n$ must be reduced. We may assume $h = gg^{-1}h = gs_1 \cdots s_n$ and hence by the left-invariance of the word metric holds
		\begin{equation*}
			d(g, h)
			= d(g, g s_1 \cdots s_n)
			= d(e, s_1 \cdots s_n)
			= n
			= d(q(g), q(h)).
		\end{equation*}
		We conclude by estimating for $g, h \in F_S$:
		\begin{align*}
			\abs{(f-\fBar)(q(g)) - (f-\fBar)(q(h))}
			&=\abs{f(q(g)) - \fBar(q(g)) - f(q(h)) + \fBar(q(h))}\\
			&= \abs{F(g) - \FBar(g) - F(h) + \FBar(h)}\\
			&\leq \norm{F - \FBar} d(g, h)
			< \eta d(q(g), q(h)).
		\end{align*}
		
		For the moreover part, we realize that for $s \in F_S$ we have
		\begin{equation*}
			c_f^+(q(s), e)
			= \sup_{g \in F_S} f(q(g)q(s)) - f(q(g))
			= \sup_{g \in F_S} (f \circ q)(gs) - f\circ q (g)
			= c_F^+(s, e).
		\end{equation*}
		Analogously we obtain $c_f^-(q(s), e) = c_F^-(s, e)$ and hence $c_f(q(s),e) = c_F(s, e)$.
	\end{proof}
	
	We have distilled the problem down to the question of whether there exists a~homomorphism on a~free group which is close to our $\delta$-invariant function on the generating set while being zero on the corresponding normal subgroup. We will demonstrate how this can be done, for the~price of worsening the constant, for finitely presented groups.
	
	\begin{lemma} \label{lma:LA_approxByKernel}
		Let $X, Y$ be Banach spaces and $A: X \to Y$ be a~finite-dimensional linear operator. Then there is a~constant $C > 0$ such that for any $x \in X$ there exists $u \in \ker A$ with $\norm{x-u} \leq C\norm{Ax}$.
	\end{lemma}
	\begin{proof}
		Put $\widetilde{X} = X / \ker A$ and $\widetilde{Y} = \im A$. Denote $q: X \to \widetilde{X}$ the quotient map and $B: \widetilde{X} \to \widetilde{Y}$ the unique linear operator satisfying $A = Bq$. The operator $B$ is bijective, so its inverse exists. Since the domain of $B^{-1}$ is the finite-dimensional space $\widetilde{Y}$, it is continuous.
		
		Fix $x \in X$. Then $B^{-1}Ax = q(x)$ and we can estimate
		\begin{equation*}
			\norm{q(x)} = \norm{B^{-1}Ax} \leq \norm{B^{-1}} \norm{Ax}.
		\end{equation*}
		By the definition of the quotient norm, for any $\eps \in (0, 1)$ there is $u_\eps \in \ker A$ such that $\norm{x - u_\eps} \leq \norm{q(x)} + \eps$. Moreover, we have
		\begin{equation*}
			\norm{u_\eps} \leq \norm{x - u_\eps} + \norm{x} \leq \norm{q(x)} + \eps + \norm{x} \leq 2\norm{x} + 1.
		\end{equation*}
		Since $\widetilde{Y}$ is finite-dimensional, the ball around its origin with radius $2\norm{x} + 1$ is compact and hence there is $u \in B_{\widetilde{Y}}(0, 2\norm{x} + 1)$ at which the continuous function $u \mapsto \norm{x-u}$ attains its minium. By what we have shown above, this minimum is at most $\norm{q(x)}$, i.e. $\norm{x - u} \leq \norm{q(x)}$. Putting the previous estimates together, we obtain
		\begin{equation*}
			\norm{x-u} \leq \norm{q(x)} \leq \norm{B^{-1}}\norm{Ax}.
		\end{equation*}
		
		We have shown the desired inequality with $C = \norm{B^{-1}}$.
	\end{proof}

	Since we need the operator $A$ to go between two general Banach spaces, we need to use this operator $B$ to be able to talk about an~inverse. However, if $A$ was a~regular matrix (that is, an~bijective operator on $\R^n$ for some $n \in \N$), we could find the sought after $u$~such that $\norm{x - u} \leq \norm{A}\norm{A^{-1}} \norm{x}$. The value $\norm{A}\norm{A^{-1}}$ for a~regular matrix $A$ is called its conditional number. The conditional number is a~fundamental concept in numerical mathematics, see e.g.~\cite[Section~3.1]{Quarteroni2007}.
	
	While the following theorem does not reflect Question~\ref{que:KYDLip} exactly, it states that if we allow the constant to depend on the group, the answer is positive for finitely-presented groups equipped with word metrics acting on themselves by translations.
	
	\begin{theorem} \label{thm:KYDfinPres}
		Let $G$ be a~finitely presented group equipped with the word metric and action by left translations. Then there exists a~constant $C > 0$ depending on $G$ such that for any $\delta > 0$ and $f \in \Lipz(G)$ $\delta$-invariant there is $\fBar \in \inv_G(G)$ with $\norm{f - \fBar} \leq C\delta$.
	\end{theorem}
	\begin{proof}
		Let $f \in \Lipz(G)$ be $\delta$-invariant. Let $S$ be the generators of $G$ and $R$ be the relators of $G$. Denote $n = \abs{S}$.
		
		We start by defining a~mapping $\#: F_S \times S \to \R^S$. Fix $g = s_{1}^{a_1} \cdots s_{m}^{a_m}$, where $m \in \N$ and $s \in S^m$, $a \in \Z^m$, an~element of  $F_S$ and its representation as a~reduced word. For $s \in S$ put $I_{g,s} = \{j \in [1..m] \colon s_{j} = s\}$ and define $\#(g, s) = \sum_{j \in I_{g,s}} a_j$. This mapping is well-defined, because the representation as a~reduced word is unique.
		
		By Proposition~\ref{prop:KYD_onFactor}, there is a~$\delta$-invariant function $F \in \Lipz(F_S)$ satisfying $F = f \circ q$. For brevity's sake, denote $c(s) = c_f(q(s), e) = c_F(s, e)$ for $s \in F_S$. Let $r \in R$ and $r = s_1^{a_1}\cdots s_k^{a_k}$, where $d(r,e) = k$, $s \in S^k$ and $a \in \{-1, 1\}^k$. Notice that $\sum_{i=1}^k a_ic(s_i) = \sum_{s \in S} \#(r, s)c(s)$. For notational convenience, denote $s_0 = e$, $a_0 = 1$. Then we have
		\begin{equation} \label{eqn:finPresAxNorm}
			\begin{split}
				\abs{\sum_{s \in S} \#(r, s)c(s)}
				&= \abs{F(r) - \sum_{s \in S} \#(r, s)c(s)}\\
				&= \abs{\sum_{i=1}^k F(s_0^{a_0} \cdots s_{i}^{a_i}) - F(s_0^{a_0} \cdots s_{i-1}^{a_{i-1}}) - \sum_{i=0}^k a_ic(s_i)}\\
				&\leq \sum_{i=1}^k \abs{F(s_0^{a_0} \cdots s_{i}^{a_i}) - F(s_0^{a_0} \cdots s_{i-1}^{a_{i-1}}) - c(s_i^{a_i})}\\
				&\leq \sum_{i=1}^k \frac{\delta}{2} d(s_i^{a_i}, e)
				= \frac{\delta}{2} k = \frac{\delta}{2} d(r,e) \leq C\delta,
			\end{split}
		\end{equation}
		where $C = \max_{r \in R} d(r, e)/2$.
		Define $A: \ell_\infty(S) \to \ell_\infty(R)$ as the (unique) linear operator satisfying $A(e_s)(r) = \#(r,s)$, $s \in S, r \in R$ and $x = (c(s))_{s \in S}$. From~\eqref{eqn:finPresAxNorm} follows that $\norm{Ax}_\infty \leq C\delta$.
		By Lemma~\ref{lma:LA_approxByKernel}, there is a~constant $D > 0$ depending only on $A$ (and hence only on $G$) and $u \in \ell_{\infty}(S)$ such that $Au = 0$ and $\supnorm{x - u} \leq D\supnorm{Ax} \leq CD\delta$.
		
		Corollary~\ref{cor:KYD_FreeAdj} applied to the function $F$ guarantees the existence of $\FBar \in \inv_{F_S}(F_S)$ with $\FBar(s) = u(s)$, $s \in S$ and $\norm{F - \FBar} \leq \delta/2 + \norm{x-u}_\infty \leq (1/2+CD)\delta$. Thanks to Proposition~\ref{prop:KYD_onFactor}, it is enough to verify that $\FBar|_{N} \equiv 0$ where $N = \ncl(R)$. Let $g \in F_S$ and $r \in R$. $\FBar$~is invariant, thus a~homomorphism by Corollary~\ref{cor:InvForLeftTranslation} and hence
		\begin{equation*}
			\FBar(g r g^{-1}) = \FBar(g) + \FBar(r) + \FBar(g^{-1}) = \FBar(r).
		\end{equation*}
		Finally, we have
		\begin{equation*}
			\FBar(r)
			= \sum_{s\in S} \#(r, s)\FBar(s)
			= \sum_{s\in S} \#(r, s) u(s)
			= \sum_{s\in S} u(s) A(e_s)(r)
			= (Au)(r)
			= 0.
		\end{equation*}
		Hence, $\FBar|_{N} \equiv 0$.
	\end{proof}
	
	Up to this point, we have been studying the word length metric. Now, we turn to the other extreme, where we begin by looking at actions of cyclic groups on themselves. Let us remark that the answer to \Cref{que:KYD} in general is known to be positive (see e.g.~\cite{Glasner22}) when the acting group is amenable. However, in the following lemma, we obtain a~stronger result which can be further applied to non-amenable groups. Lastly, in the following lemmas, we only assume to have a~pseudometric for technical reasons which will become apparent soon.
	
	\begin{lemma} \label{lma:badZ}
		Let $G$ be a~cyclic group equipped with a~left-invariant pseudometric $d$ and the action by translations on itself. If $(G, d)$ is not bi-Lipschitz equivalent to $(\Z, \abs{\cdot})$, then for every $\delta > 0$ and $\delta$-invariant $f \in \Lipz(G)$ it is the case that $\norm{f} \leq \delta$.
	\end{lemma}
	\begin{proof}
		For the cyclic (and hence commutative) group $G$ we adopt the additive notation. For the ease of notation, we assume $G$ to be either $\Z$ or $\Z_n$ for some $n \in \N$. If we have the latter, all operations are to be understood as being taken modulo $n$.
		
		First, since $G$ acts by isometries, we obtain
		\begin{equation*}
			\forall k, l \in \Z, k \geq l \colon d(k, l) = d(k-l, 0) \leq d(1, 0) + d(1, 2) + \cdots + d(k-l, k-l-1) = \abs{k-l}d(1, 0),
		\end{equation*}
		i.e. $d \leq d(0,1)\absFunc$. Hence we can without loss of generality assume that $d \leq \absFunc$ by considering the function $f/d(0,1)$ and metric $d/d(0,1)$. We assume that $d(0, 1) \neq 0$, because if $d(0, 1) = 0$, then the distance of any two points is zero and hence the only Lipschitz function on $G$ is constant zero.
		
		For a~contradiction, assume that $\norm{f}_{(\Z, d)} > \delta$. Then there are $\eps > 0$ and $l, l' \in G, l > l'$ such that $\abs{f(l) - f(l')}/d(l, l') > \delta + \eps$. Potentially passing to $-f$, we may, without loss of generality, assume that $(f(l) - f(l'))/d(l, l') > \delta + \eps$. Denote $k = l - l'$ and let $n \in \N$. Using the assumption that $\norm{gf - f} \leq \delta$ for $g = l' - nk$ we obtain
		\begin{equation*}
			\abs{(f((n+1)k) - f(nk)) - (f(l) - f(l'))}
			= \abs{(gf(l) - f(l)) - (gf(l') - f(l'))}
			\leq \delta d(l, l')
		\end{equation*}
		and hence
		\begin{align*}
			f((n+1)k) - f(nk)
			&= f(l) - f(l') - (f((n+1)k) - f(nk)) - (f(l) - f(l'))\\
			&> (\delta + \eps) d(l, l') - \delta d(l, l')
			= \eps d(l, l')
			= \eps d(k, 0).
		\end{align*}
		That is, the sequence $(f(nk))_{n=1}^\infty$ is monotone and with each step increases by at least $\eps d(k, 0)$. Thus, if we put $\alpha = \eps d(k, 0)$, then, because $f$ is Lipschitz and $f(l) \neq f(l')$, $\alpha > 0$ and 
		\begin{equation} \label{eqn:badZ}
			f(nk) > \alpha n, n \in \N.
		\end{equation}
		
		Now we discern three cases. First we deal with the easiest case when $d$ is truly a~pseudometric and not a~metric. Then, using left invariance, we have some $j \in G$ such that $d(0, j) = 0$ and hence also $d(0, jn) = 0$ for all $n \in \N$. But this is in contradiction with~\eqref{eqn:badZ}, because
		\begin{equation*}
			0 < \alpha j < f(jk) = f(jk) - f(0) \leq \norm{f} d(jk, 0) = 0.
		\end{equation*}
		
		Assume $d$ is a~metric, but $G$ is not (algebraically) isomorphic to $\Z$. Under this assumption, $\ord_G(k) < \infty$ and so, in \eqref{eqn:badZ} we may take $n = \ord_G(k)$, yielding
		\begin{equation*}
			0 = f(0) = f(\ord_G(k) k) > \alpha \ord_G(k),
		\end{equation*}
		which is, of course, a~contradiction, because, by definition of $k$, $\ord_G(k) \neq 0$.
		
		When $G$ is isomorphic to $\Z$, one must make a~slightly finer argument, starting with the following estimate
		\begin{equation} \label{eqn:KYD_ZonZ}
			\forall n \in \N \colon L(f) \geq \frac{f(nk) - f(0)}{d(nk, 0)}
			> \frac{\alpha n}{d(nk, 0)}
			= \frac{\alpha}{k} \cdot \frac{nk}{d(nk, 0)}.
		\end{equation}
		Choose $m \in \N, m > kL(f) /\alpha$. Using triangle inequality we obtain
		\begin{equation*}
			\frac{k}{\alpha} L(f)
			< m
			< \frac{k_m}{d(k_m, 0)}
			= \frac{kk_m}{kd(k_m, 0)}
			\leq \frac{kk_m}{d(kk_m, 0)}.
		\end{equation*}
		We finish by setting $n = k_m$ in~\eqref{eqn:KYD_ZonZ}:
		\begin{equation*}
			L(f) >
			\frac{\alpha}{k} \cdot \frac{k_m k}{d(k_m k, 0)}
			> \frac{\alpha}{k} \cdot \frac{k}{\alpha} L(f),
		\end{equation*}
		which is a~contradiction.
	\end{proof}
	
	\begin{lemma} \label{lma:restrictionOfDeltaInv}
		Let $G$ be a~group acting by isometries on a~pseudometric space $(M, d)$. Let $y \in M$, $H$ be a~subgroup of $G$, $g \in G$ and $f \in \Lipz(M)$ be $\delta$-invariant for some $\delta > 0$. 
		Define $d': H \times H \to \R$ by $d'(h, h') = d(hy, h'y), h, h' \in H$ and $f': H \to \R$, $f'(h) = f(ghy) - f(gy), h\in H$.
		Then $d'$ is a~left-invariant pseudometric on $H$, $f' \in \Lipz(H)$ and $f'$ is $\delta$-invariant with respect to the action by left translations on $H$.
	\end{lemma}
	\begin{proof}
		Surely, $d'$ is a~pseudometric. Let $h',h,k \in H$. Then, since $G$ acts by isometries,
		\begin{equation*}
			d'(kh, kh') = d(khy, kh'y) = d(hy, h'y) = d'(h, h'),
		\end{equation*}
		so $d'$ is in fact left-invariant.
		Clearly $f'(e) = 0$. To see that it is Lipschitz, let $h,h' \in H$ and estimate
		\begin{align*}
			\abs{f'(h) - f'(h')}
			= \abs{f(ghy) - f(gh'y)}
			\leq \norm{f}d(ghy, gh'y)
			= \norm{f}d'(h,h'),
		\end{align*}
		so $f'$ is Lipschitz with constant at most $\norm{f}$. To show that it is $\delta$-invariant for $H$, we use \Cref{lma:LipDeltaInv}~(iii). So, let $h,h',k \in H$. We have
		\begin{align*}
			\abs{f'(kh) - f'(h) - (f'(kh') - f'(h'))}
			&= \abs{f(gkhy) - f(ghy) - (f(gkh'y) - f(gh'y))}\\
			&= \abs{f(gkg^{-1}ghy) - f(ghy) - (f(gkg^{-1}gh'y) - f(gh'y))}\\
			&\leq \delta d(ghy, gh'y)
			= \delta d'(h,h').
		\end{align*}
	\end{proof}
	
	Note that the following lemma implies that, under its assumptions, the answer to \Cref{que:KYDLip} for $\delta < 1$ is trivially positive, since there are no $\delta$-invariant functions of norm $1$.
	
	\begin{lemma} \label{lma:badGroup}
		Let $G$ be a~group equipped with a~left-invariant metric $d$ and the action by translations on itself. If no cyclic subgroup of $G$ is bi-Lipschitz equivalent to $(\Z, \abs{\cdot})$, then for any $\delta > 0$ and $\delta$-invariant $f \in \Lipz(G)$ we have $\norm{f} \leq \delta$.
	\end{lemma}
	\begin{proof}
		Let $g, h \in G$, denote $H = \langle g^{-1}h \rangle$ and define $f': (H, d) \to \R$ by
		\begin{equation*}
			f'((g^{-1}h)^k) = f(g(g^{-1}h)^k) - f(g), \quad k \in \Z.
		\end{equation*}
		By \Cref{lma:restrictionOfDeltaInv} (with $y=e$) we have that $f'$ is $\delta$-invariant. Ex hypothesi, $(H, d)$ is not bi-Lipschitz equivalent to $(\Z, \abs{\cdot})$ and hence it follows from \Cref{lma:badZ} that $f'$ is $\delta$-Lipschitz. But that means that
		\begin{equation*}
			\abs{f(g) - f(h)}
			= \abs{f'((g^{-1}h)^0) - f((g^{-1}h)^1)}
			\leq \delta d((g^{-1}h)^0, (g^{-1}h)^1)
			= \delta d(g, h).
		\end{equation*}		
	\end{proof}
	
	Finally, we show that under an~additional assumption on the structure of the metric space, this result can be extended to actions not only on the group itself, but on a~larger space.
	
	\begin{theorem} \label{thm:badGroup}
		Let $G$ be a~group acting by isometries on a~pointed metric space $(M, d, 0)$. Assume the following conditions:
		\begin{enumerate}
			\item $\{g^n x : n \in \N\}$ is not bi-Lipschitz equivalent to $(\Z, \abs{\cdot})$ for every pair $g \in G, x \in X$;
			\item there is $\alpha \geq 1$ and a~subset $D \subset M$ such that $0 \in D$, $\overline{Gx} \neq \overline{Gy}$ for $x \neq y \in D$, $M = \bigcup_{x \in D} \overline{Gx}$ and $d(x, y) \leq \alpha d(\overline{Gx}, \overline{Gy}), x, y \in D$.
		\end{enumerate}
		Then for every $\delta > 0$ and $\delta$-invariant $f \in \Lipz(M)$ there is an~invariant $\fBar \in \Lipz(M)$ with $\norm{f - \fBar} \leq (2\alpha + 1)\delta$.
	\end{theorem}
	Note that for $x, y \in M$, the assertions $\overline{Gx} \cap \overline{Gy} \neq \emptyset$ and $\overline{Gx} = \overline{Gy}$ are equivalent.
	\begin{proof}
		Let $\delta > 0$ and $f \in \Lipz(M)$ be $\delta$-invariant.		
		Define a~function $\fBar: M \to \R$ by $\fBar(y) = f(x)$ for $y \in \overline{Gx}$.	By the assumption (ii) (and the note preceding the proof), this is a~well-defined function and $\fBar(0) = f(0) = 0$. To see that $\fBar$ is Lipschitz, let $x, y \in M$, find $x', y' \in D$ such that $x \in \overline{Gx'}$ and $y \in \overline{Gy'}$ and calculate
		\begin{equation*}
			\abs{\fBar(x) - \fBar(y)} = \abs{f(x') - f(y')} \leq \norm{f}d(x', y') \leq \alpha\norm{f}d(x, y).
		\end{equation*}
		Moreover, $\fBar$ is invariant which can be seen by taking $H(g) = 0$ in~\Cref{cor:LipInv} (iv) since if $x \in \overline{Gx'}$ for some $x, x' \in M$, then $Gx \subset \overline{Gx'}$.
		
		Hence, it remains to show that $\norm{f-\fBar} \leq (2\alpha + 1)\delta$. First, let $g, h \in G$ and $x, y \in D$. We have
		\begin{align*}
			\abs{(f-\fBar)(gx) - (f-\fBar)(hy)}
			&= \abs{f(gx) - f(x) - \big(f(hy) - f(y)\big)}\\
			&\leq \abs{f(gx) - f(x) - \big(f(gy) - f(y)\big)} + \abs{f(gy) - f(hy)}.
		\end{align*}
		To estimate the first summand, we can simply use the fact that $f$ is $\delta$-invariant.
		
		To get an~estimate on the second, we first define $f': G \to \R$ by $f'(g) = f(gy) - f(y), g \in G$ and $d': G \times G \to \R$ by $d'(g, h) = d(gy, hy), g,h \in G$. By \Cref{lma:restrictionOfDeltaInv} (with $g = e$ and $H = G$), we know that $d'$ is a~pseudometric and $f'$ is $\delta$-invariant Lipschitz function. From assumption (i) it follows that no cyclic subgroup of $G$ equipped with the pseudometric $d'$ is bi-Lipschitz equivalent to $(\Z, \abs{\cdot})$ and thus \Cref{lma:badGroup} yields that $f'$ is $\delta$-Lipschitz. Hence
		\begin{equation*}
			\abs{f(gy) - f(hy)}
			= \abs{f'(g) - f'(h)}
			\leq \delta d'(g, h)
			= \delta d(gy, hy).
		\end{equation*}
		
		Put together, we get
		\begin{align*}
			\abs{(f-\fBar)(gx) - (f-\fBar)(hy)}
			&\leq \delta d(x, y) + \delta d(gy, hy)\\
			&\leq \delta d(x, y) + \delta(d(gy, gx) + d(gx, hy))\\
			&\leq \delta (2\alpha + 1) d(gx, hy).
		\end{align*}
		
		Finally, if $x, y \in M$, we fix $\eps > 0$ and then find $g, h \in G$ and $x', y' \in D$ such that $d(x, gx') < \eps$ and $d(y, hy') < \eps$. Then, using the previous paragraph, we have
		\begin{align*}
			\abs{(f-\fBar)(x) - (f-\fBar)(y)}
			&\leq \abs{(f-\fBar)(gx') - (f-\fBar)(hy')} + 2\norm{f - \fBar}\eps\\
			&\leq \delta (2\alpha + 1) d(gx', hy') + 2\norm{f - \fBar}\eps\\
			&\leq \delta (2\alpha + 1) d(x, y) + 2(\norm{f - \fBar} + \delta (2\alpha + 1))\eps.
		\end{align*}
		Insomuch $\eps > 0$ was arbitrary, the proof is finished.
	\end{proof}

	A~subset $D \subset M$ satisfying $0 \in D$, $Gx \neq Gy$ for $x \neq y \in D$ and $M = \bigcup_{x \in D} Gx$ is called a~fundamental domain. A~straightforward modification of the preceding proof shows that the condition (ii) can be replaced by the condition that there is $\alpha \geq 1$ and a~fundamental domain $D \subset M$ such that $d(x, y) \leq \alpha d(Gx, Gy), x, y \in D$.

	\section{Relation to partial quasimorphisms} \label{sect:quasimorphisms}
	
	In this section, we will consider a~group $G$ equipped with an~invariant metric (that is, a~metric that is both left- and right-invariant). For such a~group, we denote $\varphi_l: G \times G \to G$ the action by left translations, i.e. $\varphi_l(g, x) = gx$, and $\varphi_r: G \times G \to G$ the action by right translations, i.e. $\varphi_r(g, x) = xg^{-1}$. Since we consider invariant metrics, both of these actions are by isometries.
	
	\begin{definition}
		Let $G$ be a~group and $f: G \to \R$. We say that $f$ is a~\emph{quasimorphism} if there is a~constant $D \geq 0$ such that
		\begin{equation*}
			\forall g, h \in G \colon \abs{f(gh) - f(g) - f(h)} \leq D.
		\end{equation*}
		If $d$ is an~invariant metric on $G$, then the mapping $f$ is said to be a~\emph{partial quasimorphism} (or $D$-partial quasimorphism) if there is a~constant $D \geq 0$ such that
		\begin{equation*}
			\forall g, h \in G \colon \abs{f(gh) - f(g) - f(h)} \leq D \min\{d(g, e), d(h, e)\}.
		\end{equation*}
	\end{definition}
	For more details and uses for quasimorphisms see, e.g.,~\cite{Calegari}.
	
	\begin{proposition} \label{prop:InvVsPQM}
		Let $(G, d)$ be a~group equipped with an~invariant metric $d$. Let $\delta > 0$ and $f \in \Lipz(G)$. Consider the statements
		\begin{enumerate}
			\item $f$ is $\frac{\delta}{2}$-partial quasimorphism;
			\item $f$ is $\delta$-invariant for the actions by translations (denoted above by $\varphi_l$ and $\varphi_r$);
			\item $f$ is $\delta$-partial quasimorphism.
		\end{enumerate}
		Then the implications $(i) \implies (ii) \implies (iii)$ hold.
	\end{proposition}
	\begin{proof}
		Notice that by \Cref{lma:LipDeltaInv}, the condition $(ii)$ is equivalent to the condition
		\begin{center}
			$(ii')$ the mappings $x \mapsto f(gx) - f(x)$ and $x \mapsto f(xg^{-1}) - f(x)$ are $\delta$-Lipschitz for all $g \in G$.
		\end{center}
		
		$(i) \implies (ii')$: Assume that $f$ is $\frac{\delta}{2}$-partial quasimorphism. Pick $g, x, y \in G$. We estimate
		\begin{align*}
			&\abs{f(gx) - f(x) - \big(f(gy) - f(y)\big)}\\
			&\hspace{5em}= \abs{f(gx) - f(gy) - f(y^{-1}x) - \big(f(x) - f(y) - f(y^{-1}x)\big)}\\
			&\hspace{5em}\leq \frac{\delta}{2} \min \{d(gy, e), d(y^{-1}x, e)\} + \frac{\delta}{2} \min\{d(y, e), d(y^{-1}x, e)\}\\
			&\hspace{5em}\leq \delta d(y^{-1}x, e)\\
			&\hspace{5em}= \delta d(x, y)
		\end{align*}
		and similarly $\abs{f(xg^{-1}) - f(x) - \big(f(yg^{-1}) - f(y)\big)} \leq \delta d(x, y)$.
		
		$(ii') \implies (iii)$:
		Using the fact that $f(e) = 0$ and the assumption $(ii')$, we obtain for any $g, h \in G$
		\begin{equation*}
			\abs{f(gh) - f(g) - f(h)}
			= \abs{f(gh) - f(h) - \big(f(ge) - f(e)\big)}
			\leq \delta d(h, e)
		\end{equation*}
		and
		\begin{equation*}
			\abs{f(gh) - f(g) - f(h)}
			= \abs{f(gh) - f(g) - \big(f(eh) - f(e)\big)}
			\leq \delta d(g, e).
		\end{equation*}
	\end{proof}
	
	We can now give an~alternative proof of one of the implications of~\cite[Theorem~1.1]{Kedra}.
	\begin{theorem}
		Let $G$ be a~group equipped with an~uniformly discrete invariant metric $d$ and $f: G \to \R$ be Lipschitz. Then $f$ is a~partial quasimorphism.
	\end{theorem}
	\begin{proof}
		Define the function $f_e \in \Lipz(G)$ by $f_e(g) = f(g) - f(e), g \in G$. Then $\norm{f_e} = L(f)$ and by \Cref{rem:everythingIsDeltaInv}, $f_e$ is $2L(f)$-invariant with respect to both actions by translations. By \Cref{prop:InvVsPQM}, $f_e$ is a~partial quasimorphism with constant $2L(f)$. Let $A > 0$ be such that for any $g \in G$ we have $A \leq d(g, e)$ and hence also $\abs{f(e)} = \frac{\abs{f(e)}}{A}A \leq \frac{\abs{f(e)}}{A}d(g,e)$. For every $g, h \in G$ we obtain
		\begin{align*}
			\abs{f(gh) - f(g) - f(h)}
			&= \abs{f_e(gh) - f_e(g) - f_e(h) - f(e)}\\
			&\leq \abs{f_e(gh) - f_e(g) - f_e(h)} + \abs{f(e)}\\
			&\leq 2L(f) \min\{d(g, e), d(h, e)\} + \frac{\abs{f(e)}}{A} \min\{d(g, e), d(h, e)\}.
		\end{align*}
		Hence, $f$ is a~partial quasimorphism with constant $2L(f) + \frac{\abs{f(e)}}{A}$.
	\end{proof}
	
	In the case that the metric is actually a~word metric given by a~generating set $S \subset G$, we can set $A = 1$ and we obtain that $f$ is a~partial quasimorphism with constant $2L(f) + \abs{f(e)}$ and $f$ is bounded on the generating set, since for any $s \in S$ we have $\abs{f(s)} \leq \abs{f(s) - f(e)} + \abs{f(e)} \leq L(f) + \abs{f(e)}$.
	
	\bibliographystyle{abbrv}
	\bibliography{bibliography}

	\include{bibliography.tex}
	
\end{document}